\documentclass{myclass}

\title{Explicit formulas for the Bessel models: odd general spin groups}
\author{Yu Xin}
\address{Department of Mathematics, Bar-Ilan University, Ramat Gan, 5290002, Israel}
\email{yu.xin@biu.ac.il}
\date{}

\begin{document}

\begin{abstract}
Let $F$ be a non-archimedean local field of characteristic zero. In this work, we study the Bessel model for $\GSpin_{2n+1}$, extending a result of Bump, Friedberg and Furusawa. In particular, we obtain explicit formulas for the unramified Bessel functions. These formulas have a global application to a Rankin--Selberg integral of the $L$-function for $\GSpin_{2n+1} \times \GL_n$, generalizing a construction of Furusawa. We compute the local factor of the global integral at a good place. Moreover, a corollary of this computation finds an application in a recent work of Asgari, Cogdell and Shahidi, specifically in their unramified computation.
\end{abstract}
\maketitle
\tableofcontents

\section{Introduction}
Let $F$ be a local non-archimedean field of characteristic zero. In this work we study the Bessel model for unramified representations of $\GSpin_{2n+1}\times\GSpin_2$. We obtain a Casselman--Shalika type formula for the normalized unramified Bessel function, by extending the similar results of Bump, Friedberg and Furusawa \cite{BFF97} for the Bessel model of $\SO_{2n+1}\times\SO_2$. As an application, we extend the global Rankin--Selberg integral of Furusawa \cite[Appendix]{BFF97} from orthogonal groups to general spin groups. Furthermore, using the methods of Soudry \cite{Sou93} and Kaplan \cite{Kap12}, our results can be utilized to compute the local Rankin--Selberg integrals of \cite{ACS24} with unramified data, for $\GSpin_{2n}\times\GL_{k}$ when $k<n$.

Let $G=\GSpin_{2n+1}(F)$ and $T=\GSpin_2(F)$. We define a Bessel subgroup of $G$ as follows. Let $P$ be a parabolic subgroup of $G$, whose Levi part is isomorphic to $\GL_1(F)\times\ldots\times\GL_1(F)\times\GSpin_3(F)$ ($\GL_1(F)$ appears $n-1$ times). Let $U$ be the unipotent radical of $P$ and let $P=U\rtimes M$ be a Levi decomposition. We can define a character $\psi$ of $U$ such that $T$ is embedded in the stabilizer of $\psi$ inside $M$. The Bessel subgroup $R$ is given by $R=U\rtimes T$.

Let $\pi$ be an unramified principal series representation of $G$, and $\lambda$ be an unramified character of $T$ such that $\lambda$ and $\pi$ agree on the center $Z_G$ of $G$. A Bessel functional on $\pi$ is by definition a non-zero functional in the space
\begin{equation}\label{eq:bessel-model-intro}
\mathrm{Hom}_{R}(\pi,\psi\otimes\lambda).
\end{equation}
This space is known to be one-dimensional for $\pi$ ``in general position'', e.g., by results of Kaplan, Lau and Liu \cite{KLL23}. 
For a Bessel functional $B$ on $\pi$, the corresponding Bessel model of $\pi$ is the space of functions $\mathcal{B}_{\xi}:G\to\C$ given by $\mathcal{B}_{\xi}(g)=B(\pi(g)\xi)$, where $\xi$ is a vector varying in the space of $\pi$. Moreover, the space in \eqref{eq:bessel-model-intro} was recently proved to be at most one-dimensional for all irreducible representations $\pi$ by Yan \cite{Yan25} (using results of \cite{ET23}, and \cite{EKM24} for the archimedean case).

We can define $B$ using an integral, which is absolutely convergent when the real part of the inducing character of $\pi$ belongs to a certain region, and elsewhere $B$ is defined by meromorphic continuation. Let $K<G$ be a maximal compact open subgroup. 
Let $\mathcal{H}$ be the unique function in the Bessel model of $\pi$ such that $\mathcal{H}(K)=1$, i.e., the 
normalized unramified function. It is straightforward to show that the value of $\mathcal{H}$ is determined by its values on torus elements $\varpi^\delta$ in the split case and on a fixed shift of
torus elements $\tilde{\varpi}^\delta$ in the quasi-split non-split case, indexed by a dominant integral vector $\delta = (\ell_1,\ldots,\ell_n)$, i.e. for such vectors with $\ell_i$ integers and with $\ell_1 \geq \ell_2 \geq \cdots \geq \ell_n \geq 0$. For the detailed definitions of the notation, see Section \ref{sec:explicit-formula}.

Our main result is the following formula, which expresses $\mathcal{H}$ on these elements in terms of the Satake parameters of $\pi$, resembling the Casselman--Shalika formula \cite{CS80}. Let us denote $\alpha_i = \chi_i(\varpi)$ where $\chi = (\chi_0,\chi_1,\ldots,\chi_n)$ is the inducing character of $\pi$ and $\varpi$ is a fixed uniformizer of $F$. Moreover, when $T = \GSpin_2(F) = Z_{G}(F) \times \SO_2(F)$ is split, we denote $\beta = \lambda(1,\diag(\varpi,\varpi^{-1}))$. Let $\calA = \sum_{w \in W} (-1)^{\text{length}(w)} w$ be the alternator in the group algebra $\C[W]$, where $W$ is the Weyl group of $\GSpin_{2n+1}$, which acts on the space of the rational functions of the Satake parameters, and let $\Delta = \calA(\alpha_1^n \cdots \alpha_n)$. For convenience, we let the quadruple $(B_\delta;z_1,z_2;Q(q))$ be $(\calH(\varpi_\delta);\alpha_0^{1/2},-\alpha_0^{1/2};1+q^{-1})$ if $T$ is quasi-split non-split and be $(\calH(\tilde{\varpi}_\delta);\beta,\alpha_0\beta^{-1};1-q^{-1})$ if $T$ is split. 

\begin{theorem}[Theorem \ref{thm:formula-nonsplit} and \ref{thm:formula-split}]\label{thm:main}
The value of $\calH$ on $G$ is determined by the formula
\[
B_\delta = \frac{\delta_{B_G}(\varpi^\delta)}{Q(q)} \Delta^{-1} \calA\left(\prod_{i=1}^n \alpha_i^{\ell_i + i} (1 - z_1 \alpha_i^{-1} q^{-\frac{1}{2}})(1 - z_2 \alpha_i^{-1} q^{-\frac{1}{2}}) \right).
\]
Here $\delta_{B_G}$ is the modulus character of the fixed Borel subgroup of $G$.
\end{theorem}

In order to obtain our formula we extend the work of Bump, Friedberg and Furusawa \cite{BFF97} who obtained the similar formula for orthogonal groups, i.e., when $G$ and $T$ are replaced with $\SO_{2n+1}(F)$ and $\SO_2(F)$, resp. Their computation was based on the Casselman--Shalika method (see \cite{Cas80,CS80}). We carry out the similar computation for general spin groups. The main difference is technical in nature: The rank of $T$ is higher than the rank of $\SO_2(F)$ by $1$, and the extra $\GL_1$ factor shows up in both the definition of the functional $B$ and the computations. 

We note that because of the exact sequence
\begin{align*}
1\to\GL_1\to\GSpin_m\to\SO_m\to1,
\end{align*}
the formula of \cite{BFF97} actually comes out as a special case of Theorem~\ref{thm:main}, so the present work can be considered a generalization of \cite{BFF97}. 

In the appendix of \cite{BFF97}, Furusawa defined a new Rankin--Selberg type integral, which represents the global partial $L$-function for automorphic cuspidal representations of $\SO_{2n+1} \times \GL_n$. The main novelty in this construction was that it did not require assuming that the cuspidal representation of the orthogonal group is globally generic (similarly general but distinct integrals were constructed in \cite{GPSR97}). The formula for the unramified Bessel function from \cite{BFF97} was used by Furusawa in order to compute the local unramified integrals arising as local components of this global integral. Subsequently, the local identities obtained by Furusawa were used by Kaplan \cite{Kap12} in order to compute the local Rankin--Selberg integrals for generic unramified representations of $\SO_{2n}\times\GL_k$ with $k<n$.

In this work we extend the global integral of Furusawa to $\GSpin_{2n+1}\times\GL_n$. The computations readily extend to general spin groups, and we also obtain, using Theorem~\ref{thm:main}, similar identities for the local unramified integrals. Recently, Asgari, Cogdell and Shahidi \cite{ACS24} constructed a Rankin--Selberg integral for globally generic automorphic representation of $\GSpin_m \times \GL_k$ for all $m$ and $k$, where they expanded the global zeta integral as an Euler product and completed the ``unramified computation'' which relates the local zeta integrals to local Langlands $L$-functions in the unramified case. Our results can be particularly used in order to complete the computation in \cite{ACS24} of the local Rankin--Selberg integrals for generic unramified representations of $\GSpin_{2n}\times\GL_k$ with $k<n$.

The Bessel model has been defined and studied for any pair of groups $\SO_m\times\SO_l$ where $m+l$ is odd, and also for hermitian groups \cite{GGP12}. The uniqueness of local Bessel models for orthogonal groups was proved by \cite{AGRS}. The extension of this model to general spin groups $\GSpin_m\times\GSpin_l$ was recently described by Yan \cite{Yan25}, who proved the uniqueness results in this setup. 
See also \cite{Sou17} for an application of Bessel models to integral representations.

As mentioned above, in order to prove our main theorem we use the Casselman--Shalika method \cite{Cas80,CS80}, following and extending the arguments of \cite{BFF97}. The recent theory of Sakellaridis \cite{Sak13} placed this method in a broader context and provided a geometric point of view.

\subsection*{Acknowledgement}
The author is deeply grateful to Eyal Kaplan for suggesting this project and for his constant support. 
The author also wishes to thank Mahdi Asgari for valuable discussions and for sharing some computations, 
and Pan Yan for sharing a manuscript of ongoing work.

This research was supported by the Israel Science Foundation (grant No. 376/21) through the author’s postdoctoral fellowship at the Department of Mathematics, Bar-Ilan University.

\section{Preliminaries}\label{sec:preliminaries}
Let $k$ be a field of characteristic zero. In this section, we recall some preliminaries on the connected split reductive group $\GSpin_{2n+1}$ over $k$.

\subsection{The root datum of $\GSpin_{2n+1}$}\label{subsec:root-datum}
There are several constructions one could give for the general spin groups. One construction is via the introduction of a based root datum for each group as in \cite[\S 7.4.1]{Spr98}, which we do below.

Let $n \geq 1$. According to \cite[\S 4]{HS16}, the based root datum of the connected split reductive group $\GSpin_{2n+1}$ is given by $(X,R,\Delta,X^\vee,R^\vee,\Delta^\vee)$, where $X$ and $X^\vee$ are $\Z$-modules generated by generators $e_0,e_1,\ldots,e_n$ and $e_0^*,e_1^*,\ldots,e_n^*$, respectively. 
The roots and coroots are given by
\begin{alignat*}{3}
    R      & = & R_{2n+1}      & = \{\pm (e_i \pm e_j): 1 \leq i < j \leq n\} \cup \{\pm e_i : 1 \leq i \leq n\} \\
    R^\vee & = & R^\vee_{2n+1} & = \{\pm (e_i^* - e_j^*): 1 \leq i < j \leq n\} 
                             \cup \{\pm (e_i^* + e_j^* - e_0^*): 1 \leq i < j \leq n\} \\ 
           &   &               & \quad \cup \{\pm(2e_i^*-e_0^*): 1 \leq i \leq n\}
\end{alignat*}
along with a bijection $R \rightarrow R^\vee$ defined by
\begin{align*}
    (\pm (e_i - e_j))^\vee &= \pm (e_i^* - e_j^*)\\
    (\pm (e_i + e_j))^\vee &= \pm (e_i^* + e_j^* - e_0^*)\\
    (\pm e_i)^\vee &= \pm (2e_i^* - e_0^*).
\end{align*}
Moreover, we fix the following choice of simple roots and coroots:
\begin{align*}
\Delta &= \{e_1-e_2, e_2-e_3, \ldots, e_{n-1}-e_n, e_n\},\\
\Delta^\vee &= \{e_1^*-e_2^*, e_2^*-e_3^*, \ldots, e_{n-1}^*-e_n^*, 2e_n^*-e_0^*\}.
\end{align*}
The based root datum determines, up to isomorphism, the group $\GSpin_{2n+1}$ together with a Borel subgroup $\borel$ and a split maximal torus $\torus \subset \borel$.

For $m \geq 1$, let $J_m$ be the antidiagonal matrix with all entries $1$.
Let $\SO_{2n+1} = \SO(J_{2n+1})$, equipped with a choice of the standard upper triangular Borel subgroup $\borel^\prime$ and the diagonal torus $\torus^\prime$. There is a projection
\begin{equation}\label{equation:projection}
\proj: \GSpin_{2n+1} \rightarrow \SO_{2n+1},
\end{equation}
whose kernel lies in the center $\mathbf{Z}$ of $\GSpin_{2n+1}$ and is isomorphic to $\GL_1$. Moreover, the projection induces isomorphisms on unipotent varieties. We can further assume the projection $\proj$ preserves the choice of Borel subgroups and split maximal torus, i.e. $$\borel = \proj^{-1}(\borel^\prime) \text{ and } \torus = \proj^{-1}(\torus^\prime).$$ 
We may, without loss of generality, represent an element 
$a$ of the maximal torus $\torus$ in the form
\[
a = e_0^*(a_0)e_1^*(a_1)\cdots e_n^*(a_n),
\]
where $a_i \in \GL_1$, and
\[
\proj(a) = \diag(a_1,\ldots,a_n,1,a_n^{-1},\ldots,a_1^{-1}).
\]

\subsection{The Weyl group of $\GSpin_{2n+1}$}\label{subsec:weyl-group}
According to \cite[Lemma 6.2.1]{HS16}, the Weyl group of $\GSpin_{2n+1}$ is canonically isomorphic to that of $\SO_{2n+1}$. We denote both Weyl groups by $W = W_{2n+1}$.

The Weyl group $W$ is canonically isomorphic to $\mathfrak{S}_n \rtimes \{\pm 1\}^n$. 
The action of $W$ on the maximal split torus $\torus$ of $\GSpin_{2n+1}$ is given by the action on the cocharacter lattice
 \begin{align*}
     (p,\epsund)\cdot e_i^* 
     &= \begin{cases}
         e_{p(i)}^* & i>0,\epsilon_{p(i)} = 1,\\
         e_0^*-e_{p(i)}^* & i>0,\epsilon_{p(i)} = -1,\\
         e_0^* & i = 0.
     \end{cases}
 \end{align*}
 
Every element $w$ of the Weyl group $W$ is represented by a permutation in $\mathfrak{S}_{2n}$, which we denote by $w$ again, such that
\[
w(i) + w(2n+2-i) = 2n+2
\]
for all $1 \leq i \leq n$.
For a more detailed discussion, we refer the reader to \cite[\S 13.2]{HS16}. 

\subsection{An accidental isomorphism}\label{subsec:acc-iso}
We let the group $\GL_2$ of invertible $2\times2$ matrices be equipped with the standard upper triangular Borel subgroup $B_{\GL_2}$ and the diagonal torus $A_{\GL_2}$. As is well known, there exists a $k$-isomorphism $\GSpin_3 \cong \GL_2$ of algebraic groups, which identifies the fixed Borel subgroup and maximal torus.

We write the based root datum of $\GL_2$ as
\(
(Y,Q,\Sigma;Y^\vee,Q^\vee,\Sigma^\vee),
\)
where
\[
Y = \Z f_1 \oplus \Z f_2, \quad 
Y^\vee = \Z f_1^* \oplus \Z f_2^*,
\]
\[
Q = \{\pm(f_1-f_2)\}, \quad 
Q^\vee = \{\pm(f_1^*-f_2^*)\},
\]
\[
\Sigma = \{f_1-f_2\}, \quad 
\Sigma^\vee = \{f_1^*-f_2^*\},
\]
together with the natural pairing on $Y \times Y^\vee$ and the bijection between roots and coroots. A torus element $\diag(t_1,t_2)\in \GL_2$ is identified with $f_1^*(t_1)f_2^*(t_2)$.

Let $X$ be as in \ref{subsec:root-datum} when $n =1 $. The linear map $X \rightarrow Y$ given by
\[
e_1 \mapsto f_1 - f_2, \quad e_0 \mapsto f_2
\]
is an isomorphism of the root data of $\GSpin_3$ and $\GL_2$. We denote the induced group isomorphism by $\iota$. In particular, $\iota$ identifies $e_1^*(t)$ with $\diag(t,1)$.

\section{The Bessel Functional}\label{sec:bessel-functional}

Let $F$ be a non-archimedean local field of characteristic zero. Let $\calO$ denote the ring of integers of $F$, $\mathfrak{p}$ denote the maximal ideal of $\calO$, $\varpi$ denote a uniformizer in $\calO$, $q$ denote the cardinality of the residue field $\calO/\varpi \calO$, and $|\cdot|_F$ denote the absolute value on $F$, normalized so that $|\varpi|_F = q^{-1}$.

Let $n \geq 1$. We adopt the notation in Section~\ref{sec:preliminaries}, taking $k = F$. Moreover, throughout Section~\ref{sec:bessel-functional}--\ref{sec:proof-mainthm}, we adopt the following convention. Let $\G = \GSpin_{2n+1}$ be defined over $F$ and let $\G^\prime = \SO_{2n+1}$ be defined over $F$ with respect to the symmetric matrix $J_{2n+1}$. If a subgroup $\H$ is defined for $\G$, we write $\H^\prime = \proj(\H)$ for its image in $\G^\prime$.
Let $\P = \M \ltimes \U$ be the standard parabolic subgroup of $\G$ where $\M \cong \GL_1^{n-1} \times \GSpin_{3}$. We note that the unipotent subgroup $\U$ is isomorphic to its image $\U^\prime$ in $\G^\prime$ via the projection $\proj$. 
In general, if $\mathbf{X}$ is an algebraic group over $F$, we write $X = \mathbf{X}(F)$ for its $F$-points. 
We use the same notation $\proj$ for the induced map of the projection \eqref{equation:projection} on $F$-points.
Finally, note that $\G$ (resp. $\G^\prime$) admits a model over $\calO$, where we denote the model by $\G$ (resp. $\G^\prime$) again.

For $S= (a,b,c)$ where $2ac+b^2 \neq 0$, we define a character $\theta_S$ on $U$ as
\[
\theta_S(u) = \psi(u_{12} + u_{23} + \ldots u_{n-2,n-1} + au_{n-1,n} + bu_{n-1,n+1} + cu_{n-1,n+2}),
\]
where $u \in U$ and $u_{ij}$ is the $(i,j)$-th entry of $\proj(u)$.
Let $\T$ be the connected component of the identity of the stabilizer, i.e.
\begin{align*}
\T := (\M^{\theta_S})^\circ &\cong \GSpin_2^\alpha
\end{align*}
where $\alpha = 2ac + b^2 \in F^\times.$ 
Here the group $\GSpin_{2}^\alpha$ is the split group $\GL_1 \times \GL_1$ if $\alpha \in F^2$, and the quasi-split non-split group $\mathrm{Res}^E_F \GL_1$ if $\alpha \in F \setminus F^2$, where $E = F(\sqrt{\alpha})$. For convenience, we will simply say $T$ is a non-split torus when it is quasi-split non-split. Indeed, a non-split form of $\GSpin_2$ is always quasi-split. (For a comprehensive introduction of $\GSpin_{2n}$ for any $n$, we refer the reader to \cite[Chapter 4]{HS16}, \cite[\S 2.2]{ACS24} and \cite{Bourbaki}.)
The \emph{Bessel subgroup} of $\GSpin_{2n+1}$ is defined to be
\(
\mathbf{R} := \T \ltimes \U.
\)
Let $\lambda$ be a character of $T$. The character $\theta_S$ is $T$-conjugation invariant and so can be extended to a generic character
\[
\theta_S: T \ltimes U \rightarrow \C^\times
\]
by letting $\theta_S(tu) = \lambda(t) \theta_S(u)$ for $t \in T$ and $u \in U$.

For $(\pi,V_\pi)$ an admissible representation of $G$, a \emph{Bessel functional} $B$ is a functional on $V_\pi$ such that
\[
B(\pi(tu) v) = \theta_S(tu) B(v) 
\]
for all $t \in T,u \in U$ and $v \in V_\pi$. The Bessel functional is unique up to scalar for an unramified principal series representation in ``general position'', this follows immediately from  \cite[Theorem 3.3]{KLL23} using transitivity of induction.
It is worth noting that the general notation for the Bessel subgroup and the Bessel functional (to be defined later) is described and studied by Yan \cite{Yan25}.

Let $\chi = (\chi_0, \chi_1, \ldots , \chi_n)$ be an unramified character on $A$. We consider the unramified principal series representation
\[
\pi = \ind(\chi).
\]
In our notation, the space $V_\pi$ consists of locally constant functions $\Psi$ on $G$ which satisfy
\[
\Psi(ang) = \delta_B^{\frac{1}{2}}(a)\left(\prod_{i=0}^n \chi_i(a_i) \right) \cdot \Psi(g)
\]
for $a= e_0^*(a_0) e_1^*(a_1) \cdots e_n^*(a_n) \in A, n \in N$ and $g \in G$, where $a_i \in \GL_1(F)$. Here $\delta_B$ is the modulus character of the Borel subgroup $B$ of $G$:
\begin{equation*}
\delta_B\left(\prod_{i=0}^n e_i^*(a_i)\right) = \prod_{i=1}^n |a_i|_F^{2n-2i+1}.
\end{equation*}
The group action is by right translation.

Suppose first that $T$ is a non-split torus. If $\chi_0 = \lambda|_Z$, we define a Bessel functional on the unramified principal series as follows.
For $\Psi \in V_\pi$, we define
\begin{equation}\label{eq:bessel-nonsplit}
B(\Psi) = \int_{T_0}\int_U \Psi(w_1ut) \theta_S^{-1}(u) du dt,
\end{equation}
where $w_1 = (1,2n+1)(2,2n)\cdots (n-1,n+3)$, $T_0 = T \cap K$ and $K = \G(\calO)$.
The absolute convergence of the integral is guaranteed if the characters $\chi_i$ are in a suitable region. We parametrize the unramified characters $\chi_i$ by $\alpha_i = \chi_i(\varpi)$ for all $i$. Indeed, since $T_0$ is compact, comparing with the standard intertwining operators $T_{w_1}$, one finds that if
\begin{equation}\label{domainC_0}
|\alpha_1| < \cdots |\alpha_{n-1}| < \min(|\alpha_n|,|\alpha_0\alpha_n^{-1}|),
\end{equation}
then the integral \eqref{eq:bessel-nonsplit} is absolutely convergent.

Suppose that $T$ is instead split. Without loss of generality, we let $S = (0,1,0)$ so that
\[
T = \{e_0^*(a_0)e_n^*(a_n)| a_0, a_n \in \GL_1\}.
\]
A character $\lambda$ of $T$ is given by a pair of multiplicative characters $\chi_0$ and $\lambda_0$ as follows:
\[
\lambda(e_0^*(a_0)e_n^*(a_n)) = \chi_0(a_0) \lambda_0(a_n).
\]
Assuming that $\lambda_0$ is unramified, a Bessel functional is given as follows.
For $\Psi \in V_\pi$, we define
\begin{equation}\label{eq:bessel-split}
B(\Psi) = \int_{Z \mo T}\int_U \Psi(w_0n(1)ut) \theta_S^{-1}(tu) du dt,
\end{equation}
where $w_0 = w_1(n,n+2)$ is a representative of the long Weyl group element and
\begin{equation}\label{eq:nx-matrix}
n(x) = \begin{pmatrix}
    I_{n-1} &&&&\\
    &1 & x & -\frac{x^2}{2}&\\
    && 1 & -x&\\
    && & 1&\\
    &&&&I_{n-1}
\end{pmatrix}.
\end{equation}
Here the Haar measure on $F^\times$ is normalized so that the measure $\calO^\times$ is $1$. Let us write $\beta = \lambda_0(\varpi)$. As we shall show below, the absolute convergence is given when
\begin{equation}\label{domainC_0prime}
|\alpha_1| < \cdots |\alpha_{n-1}| < \min(|\alpha_n|,|\alpha_0\alpha_n^{-1}|), \quad |\alpha_n| < q^{-\frac{1}{2}} \min(|\beta|,|\alpha_0\beta^{-1}|).
\end{equation}

Our first pair of results concerns the functional equations satisfied by the Bessel functional $B$. 
The Weyl group $W$ of $G$ acts on the characters $\chi = (\chi_0,\ldots,\chi_n)$, or equivalently, on the parameters $\alpha_0:=\chi_0(\varpi),\ldots,\alpha_n:=\chi_n(\varpi)$ as described in Section \ref{subsec:weyl-group}. In terms of these parameters, $W$ is the group of transformations of $(\alpha_0,\ldots,\alpha_n) \in (\C^\times)^n$ generated by
\[
(\alpha_0,\alpha_1,\ldots,\alpha_n) \mapsto (\alpha_0,\alpha_1,\ldots,\alpha_{n-1},\alpha_0\alpha_n^{-1})
\]
and by the action of the symmetric group $\mathfrak{S}_n$ on the latter $n$ variables of $(\alpha_0,\alpha_1,\ldots,\alpha_n)$.

Let $\Phi_\chi \in V_\pi$ be the standard unramified vector, which is the unique function in $V_\pi$ taking value one on $K$. 
Define the function $H_\chi$ on $G$ by
\[
H_\chi(g) = B(\pi(g)\Phi_\chi).
\]
This function is analogous to the unramified Whittaker function obtained from the standard Whittaker functional.

\begin{theorem}\label{thm:FE-nonsplit}
Suppose that $T$ is non-split. Then the function $H_\chi$, originally defined as an integral when
\[
|\alpha_1|<\ldots<|\alpha_{n-1}|< \min (|\alpha_n|,|\alpha_0\alpha_n^{-1}|),
\]
has a meromorphic continuation to all nonzero complex $\alpha_1,\ldots,\alpha_n$. Moreover the function
\[
\calH_\chi(g) = \frac{1}{\prod_{1\leq i < j \leq n}(1-\alpha_0^{-1} \alpha_i \alpha_j q^{-1})(1-\alpha_i \alpha_j^{-1} q^{-1})} H_\chi(g)
\]
is invariant under the action of $W$ on the $\alpha_i$, and holomorphic for $(\alpha_0,\alpha_1,\ldots,\alpha_n)\in (\C^\times)^n$.
\end{theorem}

\begin{theorem}\label{thm:FE-split}
Suppose $T$ is split. Then the function $H_\chi$, originally defined as an integral when
\[
|\alpha_1|<\ldots<|\alpha_{n-1}|< \min (|\alpha_n|,|\alpha_0\alpha_n^{-1}|), \quad |\alpha_n| < q^{\frac{1}{2}} \min(|\beta|,|\alpha_0\beta^{-1}|),
\]
has a meromorphic continuation to all nonzero complex $\alpha_1,\ldots,\alpha_n,\beta$. Moreover the function
\begin{equation}\label{equation:FEsplit}
\calH_\chi(g) = \frac{\prod_{i=1}^n(1-\alpha_i \beta^{-1} q^{-\frac{1}{2}})(1-\alpha_0^{-1}\alpha_i \beta q^{-\frac{1}{2}})}{\prod_{1\leq i < j \leq n}(1- \alpha_0^{-1} \alpha_i \alpha_j q^{-1})(1-\alpha_i \alpha_j^{-1} q^{-1})\prod_{i=1}^n(1-\alpha_0^{-1}\alpha_i^2 q^{-1})} H_\chi(g)
\end{equation}
is invariant under the action of $W$ on the $\alpha_i$, and holomorphic for $(\alpha_0,\alpha_1,\ldots,\alpha_n)\in (\C^\times)^n$ and $\beta$ satisfying $q^{-\frac{1}{2}} < \min(|\beta|,|\alpha_0\beta|^{-1})$.
\end{theorem}
The proofs of Theorems \ref{thm:FE-nonsplit} and \ref{thm:FE-split} are essentially the same as those for the case of odd orthogonal groups (see \cite[\S 4]{BFF97}). We highlight one main difference. To obtain the functional equation involving the action of $(n,n+2)$, \cite{BFF97} uses the isomorphism $\PGL_2 \cong \SO_3$, which identifies the inner integral in (3.8) (resp. in the last line of p.152) in \cite{BFF97} as a linear combination of right translations of the non-split (resp. split) unramified Waldspurger functions for $\ind(\chi_n,\chi_n^{-1})$. In our setting, we instead use the isomorphism $\GL_2 \cong \GSpin_3$ introduced in Section~\ref{subsec:acc-iso}, and the resulting inner integral is a Waldspurger functional for $\ind(\chi_n,\chi_0\chi_n^{-1})$. Consequently, $\alpha_0$ appears in both the domains of convergence and the normalizing factors of Theorems~\ref{thm:FE-nonsplit} and \ref{thm:FE-split}.

It is worth mentioning that the meromorphic continuation of the function $B(\Psi)$ already follows from the uniqueness result in \cite{KLL23} together with Bernstein's principle on meromorphic continuation \cite{Ban98}.

\section{Explicit formula for the Bessel model}\label{sec:explicit-formula}
As announced in the introduction, our main theorem (Theorem~\ref{thm:main}) 
provides an explicit formula for $\calH_\chi(g)$. 
We begin by introducing the necessary notation and background, after which we restate the theorem in this framework. 
The non-split and split cases will be stated separately, as their proofs require slightly different arguments.

Since
\begin{equation}\label{eq:equivariance-bessel}
\calH_\chi(rgk) = \theta_S(r) \calH_\chi(g)
\end{equation}
for $r \in R$, $g\in G$, and $k \in K$, it suffices to determine the function $\calH_\chi$ on a set of representatives for the double cosets $R \mo G /K$. Observe that the projection $\proj$ induces a bijection between the double cosets:
\begin{equation}\label{eq:double-coset-bijection}
    R \mo G /K \xlongrightarrow[]{\sim} R^\prime \mo G^\prime / K^\prime.
\end{equation}
A set of representatives for $R^\prime \mo G^\prime / K^\prime$ was described in \cite{BFF97}. Using their description and the bijection \eqref{eq:double-coset-bijection}, a set of representatives for $R \mo G /K$ can be taken as follows. Let $\delta=(\ell_1,\ldots,\ell_n)$ where $\ell_n \geq 0$. 
Let
\begin{align*}
\varpi^\delta &= \prod_{i=1}^n e_i^*(\varpi^{\ell_i}).
\end{align*}
If $T$ is non-split, $\{\varpi^\delta\}$ gives a set of coset representatives.
If $T$ is split, $\{\tilde{\varpi}^\delta\}$ gives a set of coset representatives, where
\[
\tilde{\varpi}^\delta = \begin{cases}
    \varpi^\delta & \text{ if } \ell_n = 0,\\
    n(1)\varpi^\delta & \text{ if } \ell_n > 0.
\end{cases}
\]
For convenience, we write $B_\delta$ for $\calH(\varpi^\delta)$ if $T$ is non-split and for $\calH(\tilde{\varpi}^\delta)$ if $T$ is split. 

From equation \eqref{eq:equivariance-bessel} it follows that $\mathcal{H}_\chi(g) = 0$ unless $\theta_S$ is identically one on $R \cap gKg^{-1}$.
A short calculation shows that this implies
$B_\delta = 0$ unless 
\begin{equation}\label{eq:dominant-condition}
\ell_1 \geq \ell_2 \geq \cdots \geq \ell_n \geq 0.
\end{equation}
In later sections, we say that a vector $\delta = (\ell_1,\ldots,\ell_n)$ is dominant if \eqref{eq:dominant-condition} holds.

Let $\calA$ be the alternator $\sum_{w \in W} (-1)^{\text{length}(w)} w$ an element of the group algebra $\C[W]$. Let $\Delta = \calA(\alpha_1^n \alpha_2^{n-1} \cdots \alpha_n).$
According to the Weyl denominator formula for $\GSp_{2n}(\C)$
\begin{equation}\label{eq:Delta}
    \Delta = (-1)^n\alpha_0^{\frac{n(n+1)}{2}} \prod_{i=1}^n \alpha_i^{-(n+1-i)}(1- \alpha_0^{-1}\alpha_i^2) \prod_{1 \leq i < j \leq n}(1- \alpha_0^{-1} \alpha_i \alpha_j)(1- \alpha_i \alpha_j^{-1}).
\end{equation}
This can also be easily obtained from the Weyl denominator formula for $\Sp_{2n}(\C)$ (see, for example, \cite[(1.17)]{BFF97}) by a change of variable.
Also let
\begin{align*}
e_\delta =  -\frac{1}{2} \sum_{i=1}^n \ell_i(2n+1-2i),
\end{align*}
so that $\delta_B(\varpi^\delta) = q^{e_\delta}.$
\begin{theorem}\label{thm:formula-nonsplit}
Suppose $T$ is non-split and $\delta = (\ell_1,\ldots,\ell_n)$ is dominant. Then
\[
B_\delta = \frac{q^{e_\delta}}{1+q^{-1}} \Delta^{-1} \calA\left(\prod_{i=1}^n \alpha_i^{\ell_i + i} (1-\alpha_0\alpha_i^{-2} q^{-1})\right).
\]
In particular, $\calH_\chi(I_{2n+1}) = 1$. We note that if $\ell_n = 1$, this may be written more simply as 
\begin{equation}
B_\delta = q^{e_\delta} \Delta^{-1} \calA\left(\alpha_n \prod_{i=1}^{n-1} \alpha_{i}^{\ell_i + (n+i-1)} (1-\alpha_0\alpha_i^{-2} q^{-1})\right).
\end{equation}
\end{theorem}
\begin{theorem}\label{thm:formula-split}
Suppose $T$ is split and $\delta = (\ell_1,\ldots,\ell_n)$ is dominant. Then
\begin{equation}\label{eq:k1-general}
B_\delta = \frac{q^{e_\delta}}{1 - q^{-1}} \Delta^{-1} \calA\left(\prod_{i=1}^n \alpha_i^{\ell_i + i} (1- \alpha_0\alpha_i^{-1} \beta^{-1} q^{-\frac{1}{2}})(1- \alpha_i^{-1} \beta q^{-\frac{1}{2}})\right).
\end{equation}
In particular, the function $\calH_\chi$ is holomorphic for all $(\alpha_1,\ldots,\alpha_n,\beta) \in (\C^\times)^{n+1}$, and $\calH_\chi(I_{2n+1}) = 1$.
We note that if $\ell_n = 1$, this may be written more simply as 
\begin{equation}\label{eq:k1-special}
B_\delta = q^{e_\delta} \Delta^{-1} \calA\left(\alpha_n \prod_{i=1}^{n-1} \alpha_i^{\ell_i + (n+1-i)}(1- \alpha_0 \alpha_i^{-1} \beta^{-1} q^{-\frac{1}{2}})(1- \alpha_i^{-1} \beta q^{-\frac{1}{2}})\right).
\end{equation}
\end{theorem}
\begin{remark}
Note that if $\alpha_0 = 1$, the formula in Theorem \ref{thm:formula-nonsplit} (resp. Theorem \ref{thm:formula-split}) coincides with the formula in Theorem 1.5 (resp. Theorem 1.6) in \cite{BFF97}. 
\end{remark}

\section{Proof of the formula}\label{sec:proof-mainthm}
The evaluation of the Bessel model given in Theorems \ref{thm:formula-nonsplit} and \ref{thm:formula-split} is obtained
by applying the method of Casselman and Shalika \cite{Cas80,CS80}. 
Throughout the proofs we shall assume that $\chi$ is regular, i.e. ${^w\chi} \neq \chi$ for all nonidentity $w \in W$. The general case then follows from the analytic continuation.

Let $\mathcal{B}$ be the Iwahori subgroup of $K$.
It follows from the Iwasawa decomposition of $G$ and the Bruhat decomposition over $\calO/\varpi \calO$ that the space of right-Iwahori-fixed vectors $\ind(\chi)^\calB$ is $|W|$-dimensional;
moreover, a basis is given by the functions $\phi_w$ defined by
\[
\phi_w(bw^{\prime -1}b_1) = 
\begin{cases}
    \Phi_\chi(b) & \text{ if } w = w^\prime,\\
    0 & \text{ otherwise},
\end{cases}
\]
where $b \in B$, $w \in W$, and $b_1 \in \calB$.
If $\chi$ is regular, then it is shown in \cite[\S 3]{Cas80} that the linear functionals
 \begin{align*}
     \Lambda_w: \ind(\chi)^\calB & \longrightarrow \C\\
     f &\longmapsto (T_wf)(I_{2n+1})
 \end{align*}
are linearly independent. Hence they form a basis of the dual space of $\ind(\chi)^\calB$. We denote the dual basis of those linear functionals by $\{f_w\}$.

\subsection{Proof of Theorem \ref{thm:formula-nonsplit}}Suppose $T$ is non-split. For $g \in G$, let
\[
F_\delta(g) = \int_{T_0} \int_{U_0} \Phi_\chi(gut \varpi^\delta) du dt,
\]
where $U_0 = U \cap K$.
Clearly, $F_\delta \in \ind(\chi)$. Moreover, by an argument similar to that in \cite[pp. 154--155]{BFF97}, we obtain the following lemma.
\begin{lemma}
$F_\delta \in \ind(\chi)^\calB$.
\qed
\end{lemma}
Write 
\begin{equation}\label{eq:linearrep-nonsplit}
F_\delta= \sum_{w \in W} R(\delta,w;\chi) f_w.
\end{equation}
To evaluate the coefficients, we first obtain
\begin{align}
R(\delta,w;\chi) &= T_wF_\delta(I_{2n+1}) \nonumber\\
&= \int_{N_w \mo N} \int_{T_0} \int_{U_0} \Phi_\chi(w^{-1}nut\varpi^\delta) du dt dn \nonumber\\
&= \int_{T_0} T_w\Phi_\chi(t\varpi^\delta) dt. \nonumber 
\end{align}
By \cite[Theorem 3.1]{Cas80}, we have
\[
T_w\Phi_\chi = c_w(\chi) \Phi_{^w\chi},
\]
where $\Phi_{^w\chi}$ is the standard spherical vector in $\ind({^w\chi})$, and the coefficients $c_w(\chi)$ are given by
\begin{equation}\label{eq:cw}
c_w(\chi) = \prod_{\substack{r \in R^+ \\ w(r)<0}} \frac{1-q^{-1}\chi(r^\vee(\varpi))}{1-\chi(r^\vee(\varpi))}.
\end{equation}
Hence
\[
R(\delta,w;\chi) = c_w(\chi) \sigma_{^w\chi}(\varpi^\delta),
\]
where
$$\sigma_{^w\chi}(\varpi^\delta) = \int_{T_0} \Phi_{^w\chi}(t\varpi^\delta) dt.$$
Let $\iota_n$ be the embedding given by the composition of the $F$-points of the algebraic group isomorphism $\GL_2 \rightarrow \GSpin_3$ introduced in Section \ref{subsec:acc-iso} and the standard embedding of $\GSpin_3(F)$ into $G$. Now if
\[
^w(\chi_0,\chi_1,\ldots,\chi_n) = (\chi_0^\prime,\chi_1^\prime,\ldots,\chi_n^\prime),
\]
the function $$g \mapsto \Phi_{^w \chi}(t \iota_n(g))$$ gives a $K_2$-invariant vector in $\ind(\chi_n^\prime,\chi_0 \chi_n^{\prime -1})$, where $K_2 = \GL_2(\calO)$. By the Macdonald formula \cite[Theorem 4.2]{Cas80}, we obtain
\begin{multline}\label{eq:Macdonald}
\sigma_{^w\chi}(\varpi^\delta) = \frac{q^{e_\delta}}{1 + q^{-1}} \prod_{i=1}^{n-1} \chi_{i}^\prime(\varpi^{\ell_i}) \\
\times \frac{(\chi_n^\prime(\varpi) - \chi_0\chi_n^{\prime -1}(\varpi)q^{-1})\chi_n^{\prime}(\varpi)^{\ell_n} - (\chi_0\chi_n^{\prime -1}(\varpi) - \chi_n^{\prime}(\varpi)q^{-1})\chi_0^{\ell_n}\chi_n^{\prime}(\varpi)^{-\ell_n}}{\chi_n^\prime(\varpi) - \chi_0\chi_n^{\prime -1}(\varpi)}.
\end{multline}
This completes the evaluation of $R(\delta,w;\chi)$.

Assume that $\chi$ is dominant. For $f \in \ind (\chi)$, define
\begin{equation}\label{eq:expansion-nonsplit}
L(f) = \int_U f(w_1u) \theta_S^{-1}(u) du.
\end{equation}
This integral converges absolutely by comparison with $T_{w_1}(f)$. By \eqref{eq:bessel-nonsplit} and \eqref{eq:linearrep-nonsplit},
\begin{equation}\label{eq:B-sum}
B_\delta = L(F_\delta) = \sum_{w \in W} R(\delta,w;\chi)\, L(f_w).
\end{equation}
By the above computation of $R(\delta,w;\chi)$, the contributions with $w=w_0$ and $w=w_1$ are the only ones that can produce a term of the form
\[
(\text{a rational function in } \alpha_i \text{ independent of }\delta)\times
\alpha_1^{-\ell_1}\cdots \alpha_n^{-\ell_n}.
\]

Hence it suffices to compute only for $w = w_0$ and $w_1$. We have the following lemma, similar to \cite[Lemma 4.1]{BFF97}, relating the two bases $\{\phi_w\}$ and $\{f_w\}$ for the Iwahori-fixed vectors.
\begin{lemma}
The following hold:
\begin{enumerate}
    \item $\phi_{w_0} = f_{w_0}$;
    \item $\phi_{w_1} = f_{w_1} + \frac{(1-q^{-1})\alpha_0^{-1}\alpha_n^2}{1-\alpha_0^{-1}\alpha_n^2}f_{w_0}$;
    \item $L(f_{w_0}) = 0$;
    \item $L(f_{w_1}) = 1$. 
\end{enumerate}\qed
\end{lemma}

According to \eqref{eq:cw} and \eqref{eq:Macdonald}, the expression
\begin{equation}\label{eq:term-nonsplit-beforesym}
\prod_{1 \leq i < j \leq n} \frac{1}{(1-\alpha_0^{-1}\alpha_i \alpha_j q^{-1})(1 - \alpha_i \alpha_j^{-1} q^{-1})} c_{w_1}(\chi) \sigma_{^{w_1}\chi}(\varpi^\delta)
\end{equation}
is a linear combination of rational functions of the form
\[
\alpha_1^{\pm \ell_{s(1)}} \cdots \alpha_n^{\pm \ell_{s(n)}}
\]
where $s$ is a permutation in $\mathfrak{S}_n$.
Then according to Theorem \ref{thm:FE-nonsplit}, it suffices to apply the \textit{symmetrizer} $\mathcal{S} = \sum_{w\in W} w $ to the term with $\alpha_1^{-\ell_1} \alpha_2^{-\ell_2}\cdots \alpha_n^{-\ell_n}$ in \eqref{eq:term-nonsplit-beforesym} to obtain $B_\delta$.
When $w = w_1$, we have
\[
\chi_i^\prime = 
\begin{cases}
\chi_i & {i=0 \text{ or } n},\\
\chi_0\chi_i^{-1} & {1 \leq i \leq n-1}.
\end{cases}
\]
Hence $c_{w_1}(\chi)$ equals
$$
\prod_{1 \leq i < j \leq n}\frac{(1-\alpha_0^{-1}\alpha_i\alpha_jq^{-1})(1-\alpha_i\alpha_j^{-1}q^{-1})}
{(1-\alpha_0^{-1}\alpha_i\alpha_j)(1-\alpha_i\alpha_j^{-1})} \prod_{i=1}^{n-1} \frac{1-\alpha_0^{-1}\alpha_i^2 q^{-1}}{1-\alpha_0^{-1}\alpha_i^2},
$$ 
and $\sigma_{^{w_1}\chi}(\varpi^\delta)$ equals
\begin{equation*}
\frac{q^{e_\delta}}{1+q^{-1}} \alpha_0^{\sum_{i=1}^{n-1} \ell_i} \alpha_1^{-\ell_1} \cdots \alpha_{n-1}^{-\ell_{n+1-i}}
\times \frac{(\alpha_n - \alpha_0\alpha_n^{-1}q^{-1}) \alpha_n^{\ell_n} - (\alpha_0\alpha_n^{-1} - \alpha_n q^{-1})\alpha_0^{\ell_n}\alpha_n^{-\ell_n}}{\alpha_n - \alpha_0\alpha_n^{-1}}.
\end{equation*}
We apply the symmetrize $\mathcal{S}$ to
\begin{align}\label{beforesymnon-split}
\Xi = \alpha_0^{\sum_{i=1}^n \ell_i} \alpha_1^{-\ell_1} \cdots \alpha_n^{-\ell_n} \prod_{1 \leq i < j \leq n}\frac{1}{(1-\alpha_0^{-1} \alpha_i \alpha_j)(1- \alpha_i \alpha_j^{-1})} \prod_{i=1}^{n} \frac{1-\alpha_0^{-1}\alpha_i^2 q^{-1}}{1 - \alpha_0^{-1} \alpha_i^2}
\end{align}
so that we obtain $B_\delta$.
Recalling \eqref{eq:Delta}, we have
\begin{align*}
\mathcal{S}(\Xi) &= \Delta^{-1} \calA(\Xi \Delta)\\
&= \frac{q^{e_\delta}}{1 + q^{-1}} \alpha_0^{\sum_{i=1}^n \ell_i}\alpha_0^{\frac{n(n+1)}{2}} \Delta^{-1}  \calA\left( (-1)^n \prod_{i=1}^n \alpha_{i}^{-\ell_{i} - i}(1-\alpha_0^{-1}\alpha_i^2 q^{-1})\right).
\end{align*}
Apply $w_0$ to the argument of $\calA$. We obtain
\begin{align*}
\frac{q^{e_\delta}}{1 + q^{-1}} \Delta^{-1} \calA\left(\prod_{i=1}^n \alpha_{i}^{\ell_{i} + i}(1-\alpha_0\alpha_i^{-2} q^{-1})\right).
\end{align*}
In particular, when $\ell_n = 0$, the argument of $\calA$ is
\begin{align*}
\alpha_n\prod_{i=1}^{n-1} \alpha_{i}^{\ell_{i} + i} - \alpha_0\alpha_n^{-1}q^{-1} \prod_{i=1}^{n-1} \alpha_{i}^{\ell_{i} + i}.
\end{align*}
Note that
\begin{equation}
\calA\left(-\alpha_0\alpha_n^{-1}q^{-1} \prod_{i=1}^{n-1} \alpha_{i}^{\ell_{i} + i}\right) = q^{-1} \calA\left(\alpha_n \prod_{i=1}^{n-1} \alpha_{i}^{\ell_{i} + i}\right).
\end{equation}
Hence the formula is simplified as
\begin{equation}
q^{e_\delta} \Delta^{-1} \calA\left(\prod_{i=1}^{n-1} \alpha_{i}^{\ell_{i} + i}(1-\alpha_0\alpha_i^{-2} q^{-1})\right).
\end{equation}
This completes the proof of Theorem \ref{thm:formula-nonsplit}.

\subsection{Proof of Theorem \ref{thm:formula-split}} 
Suppose $T$ is split. Let
\[
P_\delta(g) = \int_{N_0} \Phi_\chi(gn\varpi^\delta) dn,
\]
where $N_0 = N \cap K$ and $N$ is the unipotent radical subgroup of the Borel subgroup $B$ of $G$.

\begin{lemma}
$P_\delta \in \ind(\chi)^\calB$.
\end{lemma}
\begin{proof}
Clearly $P_\delta \in \ind(\chi)$. To show the right Iwahori invariance, we first recall the Iwahori factorization
\begin{equation}
    \calB = N_1^- A_0 N_0.
\end{equation}
Here we denote $N_0 = N \cap K$, $N_1 = N \cap \G(\mathfrak{p})$, $A_0 = A \cap K$ and $N_1^-$ the opposite subgroup of $N_1$. Normalizing the Haar measure on each compact group so that the volume of the group is one, we have
\begin{equation}
    \int_\calB \Phi_\chi(gk\varpi^\delta) dk = \int_{N_0} \int_{A_0} \int_{N_1^-} \Phi_\chi(gntn^-\varpi^\delta) dn^- dt dn.
\end{equation}
Since $(\varpi^\delta)^{-1} t n^- \varpi^\delta \in K$, the right-hand side is indeed $P_\delta(g)$, while the left-hand side is clearly $\calB$-invariant.
\end{proof}

Though we will ultimately use the Casselman--Shalika method, we first establish the following lemma, which implies that the determination of $B_\delta$ follows from the determination of $B(P_\delta)$. For convenience, we denote the normalizer in \eqref{equation:FEsplit} by
\[
\mathcal{N}_{\chi,\lambda} := \frac{\prod_{i=1}^{n}(1-\alpha_i \beta^{-1} q^{-\frac{1}{2}})(1- \alpha_0^{-1}\alpha_i \beta q^{-\frac{1}{2}})}{\prod_{1 \leq i < j \leq n}(1 - \alpha_0^{-1} \alpha_i \alpha_jq^{-1})(1 - \alpha_i \alpha_j^{-1}q^{-1}) \prod_{i=1}^n (1-\alpha_0^{-1} \alpha_i^2 q^{-1})}.
\]
Similar to \cite[Lemma 4.2]{BFF97}, we obtain the following lemma.
\begin{lemma}\label{lem:BH}
Let $\delta = (\ell_1,\ldots,\ell_n)$ be dominant.
\begin{enumerate}
    \item Suppose $\ell_n = 0$. Then $B_\delta = \mathcal{N}_{\chi,\lambda}^{-1}B(P_\delta)$.
    \item Suppose $\ell_n > 0$. Then
    \[
B_\delta = \frac{\mathcal{N}_{\chi,\lambda}^{-1}}{1-q^{-1}}(B(P_\delta) - q^{-1}B(P_{(\ell_1-1,\ell_2,\ldots,\ell_n)})).
    \]
\end{enumerate}\qed
\end{lemma}
We now write $P_\delta$ in terms of the Casselman basis $\{f_w\}$:
\[
P_\delta = \sum_{w \in W} S(\delta,w;\chi)f_w.
\]
The coefficients $S(\delta,w;\chi)$ are given by
\begin{align*}
S(\delta,w;\chi) 
&=\int_{N_w \mo N} \int_{N_0} \Phi_\chi(w^{-1}n n_0 \varpi^\delta) dn_0 dn\\
&=\int_{N_w \mo N} \Phi_\chi(w^{-1} n \varpi^\delta) dn\\
&= c_w(\chi) ({^w \chi }\delta^{\frac{1}{2}}_B)(\varpi^\delta).
\end{align*}
In particular, when $w = w_0$ we have
\begin{align*}
c_{w}(\chi) &= \prod_{r \in R^+} \frac{1-q^{-1}\chi(a_r)}{1-\chi(a_r)}\\
&= \prod_{1 \leq i < j \leq n}\frac{(1-\alpha_0^{-1}\alpha_i \alpha_j q^{-1})(1- \alpha_i \alpha_j^{-1} q^{-1})}{(1 - \alpha_0^{-1} \alpha_i \alpha_j)(1 - \alpha_i \alpha_j^{-1})} \prod_{i = 1}^n \frac{1-\alpha_0^{-1}\alpha_i^2 q^{-1}}{1-\alpha_0^{-1} \alpha_i^2}
\end{align*}
and
\begin{equation}\label{eq:monomial-split-presym}
{^w\chi}(\varpi^\delta) = \chi(w^{-1} \varpi^\delta w) = \alpha_0^{\sum_{i=1}^n \ell_i} \alpha_1^{-\ell_1} \cdots \alpha_n^{-\ell_n}.
\end{equation}
Since no other Weyl group element contributes to terms linearly dependent on the monomial \eqref{eq:monomial-split-presym}, it only remains to determine $B(f_{w_0})$. 

\begin{lemma}
Suppose $|\alpha_0^{-1} \alpha_n \beta|< q^{\frac{1}{2}}$. Then
    \[
B(f_{w_0}) = \left(1- \alpha_0^{-1} \alpha_n \beta q^{-\frac{1}{2}} \right)^{-1}.
    \]
\end{lemma}
\begin{proof}
Suppose $|\alpha_0^{-1} \alpha_n \beta|< q^{\frac{1}{2}}$. Note that the action of $A$ by conjugation on $N$ factors through the projection $\proj|_A$. Since $T$ normalizes $U$ and fixes $\theta_S$,
\[
B(f_{w_0}) = \int_{F^\times} \int_U f_{w_0}(w_0 n(1) e_n^*(a) u) \theta_S(u)^{-1} \lambda^{-1}(a) du d^\times a.
\]
Suppose that $w_0 n(1) e_n^*(a) u \in B w_0 \calB$. Then $(e_n^*(a^{-1}) n(1) e_n^*(a))u \in w_0 B w_0 \calB$. This relation implies that $u \in U_0$ and that $e_n^*(a^{-1}) n(1) e_n^*(a) \in K$, hence $|a|_F \geq 1$.
Thus,
\[
B(f_{w_0}) = \int_{|a|_F \geq 1} f_{w_0}(w_0n(1)e_n^*(a))\lambda^{-1}(a) d^\times a.
\]
Note that 
$$w_0 n(1) e_n^*(a) = w_0 e_n^*(a) e_n^*(a)^{-1} n(1) e_n^*(a)$$ and that the factor $e_n^*(a)^{-1} n(1) e_n^*(a)$ lies in $\calB$ for $|a|_F \geq 1$. By the right $\calB$-invariance of $f_{w_0} = \phi_{w_0}$, we have
\begin{align*}
    f_{w_0}(w_0n(1)e_n^*(a)) &= f_{w_0}(w_0 e_n^*(a) w_0^{-1}w_0)\\
    &= \chi\delta_B^{\frac{1}{2}}\left(e_0^*\left(a\right)e_n^*\left(a^{-1}\right)\right).
\end{align*}
Hence
\begin{align*}
B(f_{w_0}) &= \sum_{m=0}^\infty \alpha_0^{-m}\alpha_n^m \beta^m q^{-\frac{m}{2}} \\ &=\left(1-\alpha_0^{-1} \alpha_n \beta q^{-\frac{1}{2}}\right)^{-1}
\end{align*}
when $|\alpha_0^{-1} \alpha_n \beta|< q^{\frac{1}{2}}$.
\end{proof}

To conclude the proof of Theorem \ref{thm:formula-split}, we first observe that
\[
B(P_\delta) = \int_\calO H_\chi\left(n(x)\varpi^\delta\right) dx.
\]
Hence $B(P_\delta)$ satisfies the same functional equations with $H\left(\varpi^\delta\right)$ introduced in Theorem \ref{thm:FE-split}, i.e. $\mathcal{N}_{\chi,\lambda}^{-1}B(P_\delta)$ is invariant under the $W$-action. Similar to the non-split case, to obtain $\mathcal{N}_{\chi,\lambda}^{-1}B(P_\delta)$, we apply the symmetrize $\mathcal{S} = \sum_{w \in W} w$ to
\begin{multline*}
\frac{\prod_{i=1}^{n-1}(1-\alpha_i \beta^{-1} q^{-\frac{1}{2}})(1- \alpha_0^{-1}\alpha_i \beta q^{-\frac{1}{2}})}{\prod_{1 \leq i < j \leq n}(1 - \alpha_0^{-1} \alpha_i \alpha_j)(1 - \alpha_i \alpha_j^{-1}) \prod_{i=1}^n (1-\alpha_0^{-1} \alpha_i^2)} \\ \times q^{e_\delta} \alpha_0^{\sum_{i=1}^{n} \ell_i} \alpha_1^{-\ell_1} \cdots \alpha_n^{-\ell_n} (1- \alpha_n \beta^{-1} q^{-\frac{1}{2}}).
\end{multline*}
By a similar argument, we obtain
\begin{multline*}
q^{e_\delta} \alpha_0^{\sum_{i=1}^n \ell_i} \alpha_0^{\frac{n(n+1)}{2}} \Delta^{-1} \\
\times \calA \left((-1)^n (1- \alpha_n \beta^{-1} q^{-\frac{1}{2}}) \alpha_1^{-\ell_1 - n} \prod_{i=2}^{n} \alpha_{i}^{-\ell_i - (n+1-i)}(1-\alpha_i \beta^{-1} q^{-\frac{1}{2}})(1- \alpha_0^{-1}\alpha_i \beta q^{-\frac{1}{2}}) \right).
\end{multline*}
Applying Lemma \ref{lem:BH}, we obtain theorem \ref{thm:formula-split} with no difficulty. (Note that when $\ell_n = 0$, we obtain in fact
\begin{multline*}
B_\delta =  q^{e_\delta} \alpha_0^{\sum_{i=1}^n \ell_i} \alpha_0^{\frac{n(n+1)}{2}} \Delta^{-1} \\
\times \calA \left( (-1)^n \alpha_n^{-1} (1- \alpha_n \beta^{-1} q^{-\frac{1}{2}})  \prod_{i=1}^{n-1} \alpha_{i}^{-\ell_i - i}(1-\alpha_i \beta^{-1} q^{-\frac{1}{2}})(1- \alpha_0^{-1}\alpha_i \beta q^{-\frac{1}{2}}) \right).
\end{multline*}
Apply the action of $w_0$ to the argument of the alternator,
\[
B_\delta = q^{e_\delta} \Delta^{-1} \calA \left( \alpha_n (1- \alpha_0\alpha_n^{-1} \beta^{-1} q^{-\frac{1}{2}})  \prod_{i=1}^{n-1} \alpha_{i}^{\ell_i + i}(1-\alpha_0\alpha_i^{-1} \beta^{-1} q^{-\frac{1}{2}})(1- \alpha_i^{-1} \beta q^{-\frac{1}{2}}) \right).
\]
However, expanding the second factor, the term
\[
\alpha_n(- \alpha_0 \alpha_n^{-1} \beta^{-1} q^{-\frac{1}{2}}) \prod_{i=1}^{n-1} \alpha_{i}^{\ell_i + i}(1-\alpha_0\alpha_i^{-1} \beta^{-1} q^{-\frac{1}{2}})(1- \alpha_i^{-1} \beta q^{-\frac{1}{2}})
\]
is independent of $\alpha_n$, and hence its alternator is zero. Thus one is in fact left with \eqref{eq:k1-special}. A similar argument shows that \eqref{eq:k1-special} is equal to formula \eqref{eq:k1-general} when $\ell_n = 0$.)

\section{A Rankin--Selberg integral}\label{sec:RS-integral}
In \cite[Appendix]{BFF97}, Furusawa gave a Rankin--Selberg integral of the $L$-function for cuspidal automorphic representations of $\SO_{2n+1} \times \GL_n$, factorized the integral into an Euler product of local integrals and computed the local factors at good places. In this section, we will extend that construction to $\GSpin_{2n+1} \times \GL_n$. In particular, we will place emphasis on the local unramified computations where the explicit formulas are applied.

Let $k$ be a number field. 
For a place $v$, we let $k_v$ denote its completion at $v$. When $v$ is a finite place, we let $\calO_v$ be the ring of integers of $k_v$ and fix a uniformizer $\varpi_v$ in $\calO_v$. We also let $q_v$ be the cardinality of the residue field $\calO_v / \varpi_v \calO_v$.

Let $n \geq 1$. For $\alpha \in k^\times$, let $(V,q_V)$ be a $2n+2$-dimensional quadratic space where there is a basis $(e_1,\ldots,e_n,f_0,f_1,e_{-n},\ldots,e_{-1})$ with respect to which
the matrix of $q_V$ is
\[
\begin{pmatrix}
&&&J_n\\
&1&&\\
&&-\alpha&\\
J_n&&&\\
\end{pmatrix}.
\]

Let $H_\alpha = \GSpin(V_\alpha)$, which is a quasi-split group (note that the Witt index of $V_\alpha$ is at least $n$). 
Let $W$ be the orthogonal complement of $f_1$ in $V$, so that $\dim W = 2n+1$ and the Witt index of $W$ is $n$. Indeed, the matrix of $q_\alpha|_W$ with respect to the basis $(e_1,\ldots,e_n,f_0,e_{-n},\ldots,e_{-1})$ is $J_{2n+1}$. Let $G = \GSpin(W)$, the split form defined in Section \ref{subsec:root-datum} through its root datum. For background on $\GSpin$ from the viewpoint of quadratic spaces, see \cite{Bourbaki}, and for explicit computation of their root data, see \cite{Asg02}. The embedding of $W$ into $V$ via inclusion realizes $G$ as a subgroup of $H_\alpha$. 

Note that in this section we use italic letters, say $G$, for algebraic groups over $k$, and the same notation for their base change to $k_v$ when there is no risk of confusion. We will always specify when referring to the group of points.

Let $\A$ be the ring of adeles of $k$. Let $\pi$ be an irreducible cuspidal automorphic representation of $G_{2n+1}(\A)$ and $V_\pi$ be its space of cusp forms. We denote the central character of $\pi$ by $\chi_0$. Let $P$ be the standard parabolic subgroup with Levi factor $\GL_n \times T_\alpha$ where $T_\alpha \cong \GSpin\pmat{1}{}{}{-\alpha}$.

Let $\sigma = \bigotimes^\prime_v \sigma_v$ be an irreducible cuspidal automorphic representation of $\GL_n(\A)$ and $\lambda = \bigotimes^\prime_v \lambda_v$ be a character of $T_\alpha(k) \mo T_\alpha(\A)$ whose restriction to $Z$ is $\chi_0$.
Let $f_s \in \ind_{P(\A)}^{H_\alpha(\A)}((\sigma \otimes |\det|^s) \times \lambda)$ and let
\[
E(h,s,f_s) = L(s+1,\sigma \times \lambda) L(2s+1,\sigma,\wedge^2 \otimes \chi_0) \sum_{\gamma \in P(k) \mo H_\alpha(k)} f_s(\gamma h)
\]
be the normalized Eisenstein series attached to $f_s$. We consider the Rankin--Selberg integral given by
\begin{equation}\label{equation:RSint}
    Z(s,\phi,f_s) = \int_{G(k)\mo G(\A)} E(s,g,f_s) \phi(g) dg
\end{equation}
where $\phi \in V_\pi$.
\begin{theorem}\label{thm:appendix-main}
Suppose that the Bessel period determined by $\alpha$ and $\lambda$ does not vanish identically on $V_\pi$. (The condition will be made precise as \eqref{eq:A5-BFF} below.) Then the integral \eqref{equation:RSint} is an Euler product and its local factor at a good place $v$ equals
\begin{equation}\label{eq:appendix-main}
Z_v(s) = {L(s+1/2,\pi_v \times \sigma_v)}.
\end{equation}
The precise condition for a place $v$ to be good will be given later.
\end{theorem}
\begin{proof}
We first unfold the integral \eqref{equation:RSint}. By similar arguments to those in \cite[Appendix]{BFF97}, we have
\begin{equation}
Z(s,\phi,f_s) = L(s+1,\sigma \times \lambda)L(2s+1,\sigma,\wedge^2 \otimes \chi_0)\int_{R(\A)\mo G(\A)}W(\xi g,f_s)B_\phi(g) dg
\end{equation}
where
\[
\xi = \begin{pmatrix}
I_{n-1} & & & & & \\
&1&&&&\\
&&1&&&\\
&-1&&1&&\\
&\frac{\alpha}{2}&&-\alpha&1&\\
&&&&&I_{n-1}
\end{pmatrix}
\]
and where
\begin{equation}
B_\phi(g) = \int_{R(F) \mo R(\A)} \phi(rg) \theta_S^{-1}(r) dr.
\end{equation}
Now our principal assumption on $\pi$ is that
\begin{equation}\label{eq:A5-BFF}
B_\phi \neq 0 \text{ for some } \phi \in V_\pi.
\end{equation}
We let
\[
W(h,f_s) = \prod_v W_{s,v}(h_v).
\]
In a recent work, Yan \cite{Yan25} established the uniqueness of Bessel models over local fields of characteristic zero. As a consequence, one obtains the factorization
\[
B_\phi(g) = \prod_v B_v(g_v).
\]
Hence the Rankin--Selberg integral \eqref{equation:RSint} factorizes into an Euler product of local factors 
\begin{equation}\label{eq:A6-BFF}
Z_v(s) = L(s+1,\sigma_v \times \lambda_v)L(2s+1,\sigma_v,\wedge^2 \otimes \chi_0)\int_{R(F_v)\mo G(F_v)} W_{s,v}(\xi g) B_v(g) dg.
\end{equation}

Now we turn to the computation of the local factor at a good place $v$. More precisely, suppose that the finite place $v$ does not ramify in $F(\sqrt{\alpha})$ and $\frac{\alpha}{2}$ is a unit in $\calO_v$, the ring of integers in $F_v$. We compute \eqref{eq:A6-BFF} for unramified principal series representations $\pi_v,\sigma_v$ and an unramified character $\lambda_v$, under the assumption that $W_{s,v}$ and $B_v$ are the spherical vectors normalized so that $W_{s,v}(1) = 1$ and $B_v(1) = 1$. 

For a vector of integers $\delta = (\ell_1,\ldots,\ell_n)$, we denote 
$$\varpi^{\delta}_v = \prod_{i=1}^n e_i^*(\varpi_v^{\ell_i}).$$
By abuse of notation, we denote also by $\varpi^{\delta}_v$ a diagonal element of $\GL_n(F_v)$, $\diag(\varpi_v^{\ell_1},\ldots,\varpi_v^{\ell_n})$, when there is no fear of confusion.

By a similar argument as that in \cite[(3.4-5)]{Fur93}, using the Iwasawa decomposition and the Cartan decomposition, we have the integral over $R(F_v) \mo G(F_v)$ in \eqref{eq:A6-BFF} equals
\begin{multline}\label{eq:A7-BFF-before-nonsplit}
\sum_{\substack{\delta = (\ell_1,\ldots,\ell_{n-1},0) \\ \ell_1 \geq \cdots \geq \ell_{n-1} \geq 0}} q_v^{\sum_{i=1}^{n-1} \ell_i(2n+1-2i)} W_{s,v}(\varpi_v^\delta)B_v(\varpi_v^\delta)\\
+ (1 + q_v^{-1}) \sum_{\substack{\delta = (\ell_1,\ldots,\ell_{n}) \\ \ell_1 \geq \cdots \geq \ell_n > 0}} q_v^{\sum_{i=1}^{n} \ell_i(2n+1-2i)} W_{s,v}(\varpi_v^\delta)B_v(\varpi_v^\delta)
\end{multline}
when $v$ is inert, and equals
\begin{multline}\label{eq:A7-BFF-before-split}
\sum_{\substack{\delta = (\ell_1,\ldots,\ell_{n-1},0) \\ \ell_1 \geq \cdots \geq \ell_{n-1} \geq 0}} q_v^{\sum_{i=1}^{n-1} \ell_i(2n+1-2i)} W_{s,v}(\varpi_v^\delta)B_v(\varpi_v^\delta)\\
+ (1 - q_v^{-1}) \sum_{\substack{\delta = (\ell_1,\ldots,\ell_{n}) \\ \ell_1 \geq \cdots \geq \ell_n > 0}} q_v^{\sum_{i=1}^n \ell_i(2n+1-2i)} W_{s,v}(\varpi_v^\delta)B_v(n(1)\varpi_v^\delta)
\end{multline}
when $v$ is split. We recall that $n(1)$ is defined in \eqref{eq:nx-matrix}.

Note $\delta_P(\iota_\alpha(\varpi_v^\delta)) =|\det(\varpi_v^\delta)|^{n+1}$,
where $\delta_P$ is the modulus character of the parabolic subgroup $P(F_v)$ of $H_\alpha(F_v)$, so
\[
W_{s,v}(\varpi_v^\delta) = q_v^{-(s+\frac{n+1}{2}) \sum_{i=1}^n \ell_i} W_v^\sigma(\varpi_v^\delta),
\]
where $W^\sigma$ is the spherical Whittaker function of the unramified principal series $\sigma_v$, normalized as $W^\sigma(1) = 1$. We further apply the Casselman--Shalika formula \cite{Shi76,CS80}
\begin{align}\label{eq:CS-formula-in-character}
W_v^\sigma(\varpi_v^\delta) &= \delta_{B_{\GL_n}}^{\frac{1}{2}}(\varpi_v^\delta) X_\delta^{\GL_n}(t_{\sigma,v})
\end{align}
and Theorems \ref{thm:formula-nonsplit} and \ref{thm:formula-split}
\begin{align}\label{eq:my-formula-in-Sdelta}
B_v^\pi(\varpi_v^\delta) &= \delta_B^{\frac{1}{2}}(\varpi_v^\delta) S_\delta(t_{\pi,v},t_{\lambda,v}),
\end{align}
where:
\begin{itemize}
\item $t_{\sigma,v}$ is the Satake parameter of $\sigma_v$;
\item for $m \geq 1$, $\delta_{B_{\GL_m}}$ denotes the modulus character of the Borel subgroup $B_{\GL_m}$ of $\GL_m$;
\item for $m \geq 1$ and an integral vector $\kappa = (t_1,\ldots,t_m)$ with $t_1 \geq \cdots \geq t_m$, let $X_\kappa^{\GL_m}$ denote the character of the irreducible finite-dimensional holomorphic representation of $\GL_m(\C)$ of highest weight $\kappa$; 
\item $t_{\pi,v} = \diag(\alpha_1,\ldots,\alpha_n,\alpha_0\alpha_n^{-1},\ldots,\alpha_0\alpha_1^{-1})$ and $t_{\lambda,v} = \diag(z_1,z_2)$ are the Satake parameters of $\pi_v$ and $\lambda_v$, respectively, and
\[
S_\delta(t_{\pi,v},t_{\lambda,v}) :=
   \frac{1}{Q(q_v)\Delta}\,
   \calA\!\left(\prod_{i=1}^n
       \alpha_i^{\ell_i+(n+1-i)}
       (1- z_1 \alpha_i^{-1} q_v^{-\frac{1}{2}})
       (1- z_2 \alpha_i^{-1} q_v^{-\frac{1}{2}})\right);
\]
\item the triple $(z_1,z_2;Q(q_v))$ is given by
\[
(z_1,z_2;Q(q_v)) =
\begin{cases}
(\beta,\, \alpha_0\beta^{-1};\, 1-q_v^{-1}), & \text{if $T_\alpha$ is split},\\[6pt]
(\alpha_0^{\frac{1}{2}},\, -\alpha_0^{\frac{1}{2}};\, 1+q_v^{-1}), & \text{if $T_\alpha$ is non-split}.
\end{cases}
\]
\end{itemize}
The formulas \eqref{eq:A7-BFF-before-nonsplit} and \eqref{eq:A7-BFF-before-split} are simplified as
\begin{multline}\label{eq:A7-BFF} 
\sum_{\substack{\delta = (\ell_1,\ldots,\ell_n) \\ \ell_1 \geq \cdots \geq \ell_{n-1} \geq \ell_n = 0}}  q_v^{-(s+\frac{1}{2})\sum_{i=1}^n \ell_i} X_\delta^{\GL_n}(t_{\pi,v}) S_\delta(t_{\pi,v},t_{\lambda,v})\\
+ Q(q_v) \sum_{\substack{\delta = (\ell_1,\ldots,\ell_n) \\ \ell_1 \geq \cdots \geq \ell_{n-1} \geq \ell_n > 0}} q_v^{-(s+\frac{1}{2})\sum_{i=1}^n \ell_i} X_\delta^{\GL_n}(t_{\pi,v}) S_\delta(t_{\pi,v},t_{\lambda,v}).
\end{multline}
Let the Satake parameter $t_{\sigma,v}$ be
$$t_{\sigma,v} = \diag(\gamma_1,\ldots,\gamma_n) \in \GL_n(\C).$$
The Weyl group of $\GL_n$, which we denote by $W_{\GL_n}$, acts on $\C[\gamma_1,\ldots,\gamma_n]$. Similar to the definition of the alternator $\calA$ in $\C[W]$, we let $\calB = \sum_{w \in W_{\GL_n}} (-1)^{\text{length}(w)}w$ be the alternator in the group algebra $\C[W_{\GL_n}]$. By Weyl's character formula for $\GL_n(\C)$,
\begin{align}
X_\delta(t_{\sigma,v}) = \Delta^{-1}_{\GL_n} \calB( \gamma_1^{\ell_1+n-1}\gamma_2^{\ell_2+n-2}\cdots \gamma_n^{\ell_n})\label{eq:Whittaker}
\end{align}
where $\Delta_{\GL_n} = \calB(\gamma_1^{n-1} \cdots \gamma_{n-1})$. To be consistent with the notation $\Delta_{\GL_n}$, we write $\Delta_{\GSp_{2n}}$ for $\Delta$ in the rest of this paper.
We regard $\calA$ and $\calB$ as commuting operators on
\[
\C[\alpha_0,\alpha_1^{\pm 1},\ldots,\alpha_n^{\pm 1}] \otimes \C[\gamma_1,\ldots,\gamma_n].
\]
Let $\mu_{\text{sim}}$ be the similitude character of $\GSp_{2n}(\C)$. For an element $g \in \GSp_{2n}(\C)$, let $\bar{g} = \mu_{\text{sim}}(g)^{-1/2}g \in \Sp_{2n}(\C)$. In particular,
$\mu_{\text{sim}}(t_{\pi,v}) = \alpha_0$ and
\[
\bar{t}_{\pi,v} = \diag(\alpha_0^{-\frac{1}{2}}\alpha_1,\ldots,\alpha_0^{-\frac{1}{2}}\alpha_n,\alpha_0^{\frac{1}{2}}\alpha_n^{-1},\ldots,\alpha_0^{\frac{1}{2}}\alpha_1^{-1}).
\]
According to \cite[(8.27) and (8.41)]{ACS24}, we have
\[
L(s+1/2,\pi_v \times \sigma_v) = L(2s+1,\sigma_v, \wedge^2 \otimes \chi_{0,v}) D(\alpha,\gamma;X)
\]
where $X = q_v^{-(s+\frac{1}{2})}$, $X_\delta^{\Sp_{2n}}$ is the character of the finite-dimensional representation of $\Sp_{2n}(\C)$ with respect to highest weight $\delta$, and
\begin{align}
D(\alpha,\gamma;X)
&= \sum_{\ell_1 \geq \cdots \geq \ell_n \geq 0} \alpha_0^{\frac{1}{2}\sum_{i=1}^n \ell_i} X_\delta^{\GL_n}(t_{\sigma,v}) X_{\delta}^{\rm Sp_{2n}}(\bar{t}_{\pi,v}) X^{(\ell_1 + \cdots + \ell_n)}.
\end{align}
By Weyl's character formula for $\Sp_{2n}(\C)$,
\begin{multline*}
X_\delta^{\Sp_{2n}}(\bar{t}_{\pi,v}) = \frac{\calA\left(\prod_{i=1}^n (\alpha_0^{-\frac{1}{2}}\alpha_i)^{\ell_i + n+1-i}\right)}{\calA\left(\prod_{i=1}^n (\alpha_0^{-\frac{1}{2}}\alpha_i)^{n+1-i}\right)}
= \alpha_0^{-\frac{1}{2}\sum_{i=1}^n \ell_i} \Delta_{\GSp_{2n}}^{-1} \calA\left(\prod_{i=1}^n \alpha_i^{\ell_i + n + 1 - i}\right).
\end{multline*}
Hence
\begin{align}
D(\alpha,\gamma;X) = \sum_{\ell_1 \geq \cdots \geq \ell_n \geq 0} \Delta_{\GL_n}^{-1} \Delta_{\GSp_{2n}}^{-1} \calA\calB\left(\alpha_1^n \cdots \alpha_n^1 \gamma_1^{n-1}\cdots \gamma_{n-1}\prod_{i=1}^n (\alpha_i\gamma_i)^{\ell_i}\right) X^{(\ell_1 + \cdots + \ell_n)}.\label{eq:D-alpha-gammaX}
\end{align}
Note that 
\begin{equation}
L(s+1,\sigma_v \times \lambda_v) = \prod_{i=1}^n \frac{1}{(1- z_1 \gamma_i  q_v^{-\frac{1}{2}}X)(1- z_2 \gamma_i q_v^{-\frac{1}{2}}X)}.
\end{equation}
In order to prove \eqref{eq:appendix-main}, it suffices to show that
\begin{align}
D&(\alpha,\gamma;X) \prod_{i=1}^n (1- z_1\gamma_i q_v^{-\frac{1}{2}}X)(1- z_2\gamma_i q_v^{-\frac{1}{2}}X) = \label{eq:A8-BFF-nonsplit}\\
\sum_{\substack{\ell_1 \geq \cdots \geq \ell_n\\ \ell_n = 0}} & \Delta^{-1}_{\GL_n} \Delta^{-1}_{\GSp_{2n}} X^{\ell_1 + \cdots + \ell_n}\nonumber\\
\times \calA\calB&\left(\alpha_1^n \alpha_2^{n-1} \cdots \alpha_n \gamma_1^{n-1}\gamma_2^{n-2}\cdots\gamma_{n-1}\prod_{i=1}^{n-1}(\alpha_i \gamma_i)^{\ell_i}(1- z_1 \alpha_i^{-1} q^{-\frac{1}{2}})(1- z_2 \alpha_i^{-1} q^{-\frac{1}{2}})\right)\nonumber\\
+ \sum_{\substack{\ell_1 \geq \cdots \geq \ell_n\\ \ell_n > 0}}&\Delta^{-1}_{\GL_n} \Delta^{-1}_{\GSp_{2n}} X^{\ell_1 + \cdots + \ell_n}\nonumber\\
\times \calA\calB&\left(\alpha_1^n \alpha_2^{n-1} \cdots \alpha_n \gamma_1^{n-1}\gamma_2^{n-2}\cdots\gamma_{n-1}\prod_{i=1}^{n}(\alpha_i \gamma_i)^{\ell_i}(1- z_1 \alpha_i^{-1} q^{-\frac{1}{2}})(1- z_2 \alpha_i^{-1} q^{-\frac{1}{2}})\right).\nonumber
\end{align}
Here we note that the right-hand side of \eqref{eq:A8-BFF-nonsplit} equals \eqref{eq:A7-BFF} from the above discussion. We first show
\begin{multline}\label{eq:BFG1}
D(\alpha,\gamma;X)\prod_{i=1}^n (1 - z_1 \gamma_i q^{-\frac{1}{2}}X) = \sum_{\ell_1 \geq \cdots \geq \ell_n \geq 0} \Delta^{-1}_{\GL_n} \Delta^{-1}_{\GSp_{2n}} X^{\ell_1 + \cdots + \ell_n}\\
\times \calA\calB\left(\alpha_1^n \alpha_2^{n-1} \cdots \alpha_n \gamma_1^{n-1}\gamma_2^{n-2}\cdots\gamma_{n-1}\prod_{i=1}^{n}(\alpha_i \gamma_i)^{\ell_i}(1- z_1 \alpha_i^{-1} q^{-\frac{1}{2}})\right).
\end{multline}
First note that 
\(
\prod_{i=1}^n (1- z_1 \gamma_i q_v^{-\frac{1}{2}}) = \sum_{S \subseteq \NN{n}} \prod_{i=1}^n (- z_1 \gamma_i q_v^{-\frac{1}{2}})^{\chi_S(i)},
\)
where $\chi_S$ is the characteristic function of the set $S$. The left-hand side of \eqref{eq:BFG1} is then
\begin{multline*}
\sum_{S \subseteq \NN{n}} \sum_{\ell_1 \geq \cdots \geq \ell_n \geq 0} X^{\ell_1+ \cdots +\ell_n} \\
\times \calA \calB \left(\AAA\GGG \prod_{i=1}^n (\alpha_i \gamma_i)^{\ell_i} \prod_{i=1}^n (-z_1\gamma_i q_v^{-\frac{1}{2}}X)^{\chi_S(i)}\right).
\end{multline*}
By a change of variables $\ell_i$ to $\ell_i - \chi_S(i)$, we obtain
\begin{multline}\label{eq:before-claim}
\sum_{S \subseteq \NN{n}} \sum_{\ell_1-\chi_S(1) \geq \cdots \geq \ell_n-\chi_S(n) \geq 0} X^{\ell_1+ \cdots +\ell_n}\\
\times \calA \calB \left(\AAA\GGG \prod_{i=1}^n (\alpha_i \gamma_i)^{\ell_i} \prod_{i=1}^n (- z_1 \alpha_i^{-1} q_v^{-\frac{1}{2}})^{\chi_S(i)}\right).
\end{multline}
\begin{claim}
The condition $\ell_1 -\chi_S(1) \geq \cdots \geq \ell_n-\chi_S(n) \geq 0$ in \eqref{eq:before-claim} can be replaced by $\ell_1 \geq \cdots \geq \ell_n \geq 0$ without changing its value.
\end{claim}
\begin{proof}[Proof of Claim]
We first replace $\ell_i - \chi_S(i) \geq \ell_{i+1} - \chi_S(i+1)$ for $i = 1,2,\ldots,{n-1}$. 
If $i,i+1 \in S$ or $i,i+1 \notin S$, it is obvious that the condition $\ell_i - \chi_S(i) \geq \ell_{i+1} - \chi_S(i+1)$ is directly equivalent to $\ell_i \geq \ell_{i+1}$. If $i \in S$ and $i + 1 \notin S$, the condition is  $\ell_i >\ell_{i+1}$. However, when $\ell_i = \ell_{i+1}$, the corresponding summand vanishes since the exponents of $\alpha_i$ and $\alpha_{i+1}$ in the argument of $\calA\calB$ are equal. Similarly, when $i \notin S$ and $i+1 \in S$, the condition is $\ell_i \geq \ell_{i+1} -1$, and the corresponding summand for $\ell_i = \ell_{i+1} - 1$ vanishes since the exponents of $\gamma_i$ and $\gamma_{i+1}$ in the argument of $\calA\calB$ are equal.

Now we replace $\ell_n \geq \chi_S(n)$ by $\ell_n \geq 0$. Indeed, we notice that when $n \in S$ and $\ell_n = 0$, the argument of $\calA\calB$ is independent of $\alpha_n$.
\end{proof}

By \eqref{eq:BFG1}, the left-hand side of \eqref{eq:A8-BFF-nonsplit} equals
\begin{multline*}
\prod_{i=1}^n (1 -  z_2 \gamma_i q^{-\frac{1}{2}}X )\sum_{\ell_1 \geq \cdots \geq \ell_n \geq 0} \Delta^{-1}_{\GL_n} \Delta^{-1}_{\GSp_{2n}} X^{\ell_1 + \cdots + \ell_n}\\
\times \calA\calB\left(\alpha_1^n \alpha_2^{n-1} \cdots \alpha_n \gamma_1^{n-1}\gamma_2^{n-2}\cdots\gamma_{n-1}\prod_{i=1}^{n}(\alpha_i \gamma_i)^{\ell_i}(1- z_1 \alpha_i^{-1} q_v^{-\frac{1}{2}})\right).
\end{multline*}
Next, we employ an argument analogous to that used in the proof of \eqref{eq:BFG1}, and obtain
\begin{multline*}
\sum_{S \subseteq{\NN{n}}} \sum_{\ell_1 \geq \cdots \geq \ell_n \geq \chi_S(n)} \Delta^{-1}_{\GL_n} \Delta^{-1}_{\GSp_{2n}} X^{\ell_1 + \cdots + \ell_n}\\
\times \calA\calB\left(\alpha_1^n \alpha_2^{n-1} \cdots \alpha_n \gamma_1^{n-1}\gamma_2^{n-2}\cdots\gamma_{n-1}\prod_{i=1}^{n}(\alpha_i \gamma_i)^{\ell_i}(1- z_1 \alpha_i^{-1} q_v^{-\frac{1}{2}}) \prod_{i=1}^n (-z_2 \gamma_i q_v^{-\frac{1}{2}}X)^{\chi_S(i)}\right)
\end{multline*}
with no difficulty. This equals the right-hand side of \eqref{eq:A8-BFF-nonsplit},
which completes the proof of Theorem \ref{thm:appendix-main}.
\end{proof}

We actually proved a formal identity that \eqref{eq:A7-BFF} equals
\[
\frac{L(s+1/2,\pi_v \times \sigma_v)}{L(s+1,\sigma_v \times \lambda_v)L(2s+1,\sigma_v,\wedge^2 \otimes \chi_{0,v})}.
\]
Following a similar argument to that of \cite[Corollary 2.1]{Kap12}, where the technique of Ginzburg's \cite{Gin90} for reducing Satake parameter is applied, we have
\begin{corollary}
Let $F$ be a non-archimedean local field of characteristic zero. Let $1 \leq l < n$ and denote $\sigma^\prime$ (resp. $\pi$, $\lambda$) an unramified principal series representation of $\GL_l(F)$ (resp. ($\GSpin_{2n+1}(F)$, $\GSpin_2(F)$) with Satake parameter $t_{\sigma^\prime}$ (resp. $t_\pi$, $t_\lambda$). The following identity holds:
\begin{equation*}
\sum_{\substack{\ell_1 \geq \cdots \geq \ell_l \geq 0}}  X_\delta^{\GL_l}(t_{\sigma^\prime}) S_\delta(t_\pi,t_\lambda) q^{-(s+\frac{1}{2})\sum_{i=1}^{l} \ell_i} = \frac{L(s+1/2,\pi \times {\sigma^\prime})}{L(s+1,{\sigma^\prime} \times \lambda)L(2s+1,{\sigma^\prime},\wedge^2 \otimes \chi_0)}.
\end{equation*}
Here $\GSpin_2(F)$ is split or quasi-split non-split.\qed
\end{corollary}

\end{document}